\documentclass[11pt]{amsart}

\usepackage{epsfig,graphicx,color,mathrsfs}
\usepackage{graphicx}
\usepackage{amsmath,amssymb,amsthm,amsfonts}
\usepackage{amssymb}
\usepackage[english]{babel}
\usepackage{epsfig,graphicx,color,mathrsfs}
\usepackage[english]{babel}
\usepackage[left=2.7cm,right=2.7cm,top=3cm,bottom=3cm]{geometry}
\usepackage{enumitem}
\newtheorem{theorem}{Theorem}[section]
\newtheorem{proposition}[theorem]{Proposition}
\newtheorem{lemma}[theorem]{Lemma}

\numberwithin{equation}{section}

\newcommand{\R}{\mathbb R}
\newcommand{\eps}{\varepsilon}
\newcommand{\e}{\varepsilon}

\newcommand{\n}{\nabla}
\newcommand{\la}{\langle}
\newcommand{\ra}{\rangle}
\newcommand{\vf}{\varphi}

\newcommand{\neps}{(\varepsilon+|\nabla u_\varepsilon|^2)}
\newcommand{\io}{\int_\Omega}
\newcommand{\uei}{u_{\varepsilon,i}}
\newcommand{\ue}{u_\varepsilon}

\title[Regularity]{A regularity result for the p-laplacian near uniform ellipticity}

\allowdisplaybreaks

\author{Carlo Mercuri}
\address{Swansea University\\ Department of Mathematics\\ Singleton Park\\
Swansea\\ SA2~8PP\\ Wales, United Kingdom}
\email{C.Mercuri@swansea.ac.uk}

\author{Giuseppe Riey}
\address{Universit\`a della Calabria\\
Dipartimento di Matematica e Informatica\\
Pietro Bucci 31B, I-87036 Arcavacata di Rende, Cosenza, Italy}
\email{Riey@mat.unical.it}

\author{Berardino Sciunzi}
\address{Universit\`a della Calabria\\
Dipartimento di Matematica e Informatica\\
Pietro Bucci 31B, I-87036 Arcavacata di Rende, Cosenza, Italy}
\email {Sciunzi@mat.unical.it}

\begin{document}
\begin{abstract}
 We consider weak solutions to a class of Dirichlet boundary value problems involving the $p$-Laplace operator, and prove that the second weak derivatives are in $L^{q}$ with $q$ as large as it is desirable, provided $p$ is sufficiently close to $p_0=2$. We show that this phenomenon is driven by the classical Calder\'on-Zygmund constant. As a byproduct of our analysis we show that $C^{1,\alpha}$ regularity improves up to $C^{1,1^-}$, when p is close enough to 2. This result we believe it is particularly interesting  in higher dimensions $n>2,$
when optimal $C^{1,\alpha}$ regularity is related to the optimal regularity of $p$-harmonic mappings, which is still open (see e.g. \cite{Te}).
 \end{abstract}


\maketitle

\tableofcontents

\section{Introduction and results}\label{introdue}

In this paper we deal with the $W^{2,q}$ regularity of the weak solutions to
\begin{equation}\label{EQ}
\begin{cases}
-\Delta_p\,u=f, & \text{ in }\Omega\\
u=0,&  \text{ on }\partial\Omega\,,
\end{cases}
\end{equation}
where $p>1$,  $\Delta_p u:={\rm div}(|\nabla u|^{p-2}\nabla u)$ is the $p$-Laplace operator, $\Omega$ is a bounded smooth domain of $\mathbb{R}^n,n\geq2,$ and

$$
f\in \left\{
\begin{array}{ll}
W^{1,r}(\Omega),\,\,\,r\in (n, \infty],  & \quad \hbox{ if } p > 2 , \\

C(\overline{\Omega}),& \hbox{ if } p\leq 2.
\end{array}
\right.
$$

Namely, we consider possibly sign-changing functions $u\in W^{1,p}_{0}(\Omega)$ such that
\begin{equation}\label{regreggete0}
\int_\Omega |\nabla u|^{p-2}\nabla u\,\nabla \psi\,dx\,=\,\int_\Omega f\psi \,dx\,,
\end{equation}
for all $\psi\in W^{1,p}_{0}(\Omega)$. \\
It is well-known that under our assumptions $u\in C^{1,\alpha}(\overline{\Omega}),$ for some $\alpha<1,$ as it follows by the classical results \cite{Di,Li,T}. Furthermore, results
 on the optimal H\"{o}lder exponent $\alpha$ are also known in the literature, see e.g. the recent paper \cite{Te} that is also based on previous results obtained in \cite{kusi}.\\
It is worth mentioning that, since by classical Morrey's embedding theorem our function $f$ is (up to the boundary H\"older) continuous, the notion of weak and viscosity solutions (see e.g. \cite{CIL}) coincide. This follows by the result of \cite{lind}; the recent paper \cite{julin} contains a new interesting proof of this known fact.\\
\noindent Here we address the study of the summability of the second derivatives of the solutions on the whole $\Omega$. From \cite{DS1}, \cite{Sci1,Sci2}) it is known under the above assumptions that $u\in W^{2,2}_{\textrm{loc}}(\Omega)$ if $1<p<3$,  and that if $p\geq 3$ and the source term $f$ is strictly positive then $u\in W^{2,q}_{\textrm{loc}}(\Omega)$ for $q<\frac{p-1}{p-2}$. We observe that, in the case $p\geq 3,$ it is possible to construct examples which show that such a regularity is optimal, see e.g. \cite{Sci2}.
The above-mentioned regularity results are obtained exploiting improved weighted estimates on the summability of the second derivatives, see e.g. \cite{T} . Note that the Calder\'on-Zygmund theory cannot be extended trivially to the context of quasilinear elliptic problems. The interested reader is referred to \cite{surveymin,ZIGMUND} and the references therein. \\
\noindent As a preliminary observation we first would like to point out that the aforementioned weighted estimates on the second derivatives holds up to the boundary in those cases where one can handle problems caused by the intersection of the critical set $Z_u=\{x\in \overline{\Omega}\,\,: \,\,\nabla u =\bf{0}\}$ with $\partial \Omega.$
We have the following
\begin{proposition}\label{Hesspart}
Let  $p>1$ and let $\Omega$ be a bounded smooth domain of $\mathbb{R}^n,$ and $u$ be weak solution solution of (\ref{EQ}) with $\nabla u\neq \bf{0}$ on $\partial \Omega$. Assume that $f\in W^{1,r}(\Omega)$ for some $r\in (n, \infty]$ and for any $p>1$.
Then, for any $\beta<1$ there exists a constant $C_\beta=C_\beta(n,p,f)>0$
such that
\begin{equation}\label{kdjlgadklfgbaò}
\int_{\Omega} |\nabla u|^{p-2-\beta}  \|D^2 u\|^2 <C_\beta.
\end{equation}

\end{proposition}
Here $D^2 u$ denotes the Hessian matrix of $u$, and $\|D^2 u\|$ is any equivalent norm of it. This implies in particular the global regularity $u\in W^{2,2}(\Omega)$ if $1<p<3$. Note that no sign assumption on $u$ nor on the source term $f$ is required.
\begin{proof}
By a classical result of G. Stampacchia it is well-known that the second derivatives of $u$ vanish almost everywhere on the critical set $Z_u.$ From Corollary 2.1 \cite{DS1} (see also \cite{Sci1,Sci2}) we have \begin{equation}\label{kkdjlgadklfgbaò}
\int_{\omega} |\nabla u|^{p-2-\beta}  \|D^2 u\|^2 <\infty
\end{equation}
on every open set $\omega \subset \Omega$ which is strictly contained in $\Omega$. This estimate holds in fact up to the boundary. This can be proved word by word as in Corollary 2.1 in \cite{DS1} simply replacing in the proof Hopf's boundary point lemma by our boundary assumption on the gradient. And this concludes the proof.
\end{proof}
Note that, as we will discuss later on, the general assumption $\nabla u\neq \bf{0}$ on $\partial \Omega$ is fulfilled e.g. in all those cases when the Hopf boundary lemma applies (see \cite{V}).
We are now able to highlight a regularising effect that occurs when the quasilinear equation approaches the semilinear one, namely when $p$ approaches $p_0=2$. This phenomenon is driven by the classical elliptic regularity theory, namely by the constant $C(n,q)$ in the Calder\'on-Zygmund estimate
\begin{equation}\label{calderon0}
\|D^2 w\|_{L^q(\Omega)}\leq C(n,q)\|\Delta w\|_{L^q(\Omega)},\,
\end{equation}
see e.g. Corollary 9.10 in \cite{GiTr}. We point out that it is because of this result that we restrict ourselves to solutions which vanish on $\partial \Omega$. \\
To describe this phenomenon let us start recalling that, formally we have
\begin{equation}\nonumber
\Delta _p u\,=\, |\nabla u|^{p-2}\Delta u+(p-2)|\nabla u|^{p-4}\Delta_\infty u
\end{equation}
where $\Delta_\infty u = \left(D^2 u \nabla u,\nabla u\right)$. Still formally, using the above decomposition, we can rewrite our equation as
\begin{equation}\nonumber
-\Delta u\,=\, (p-2)\frac{\Delta_\infty u}{|\nabla u|^{2}}
\,+\,\frac{f}{|\nabla u|^{p-2}}.
\end{equation}
Then we show that the term $(p-2)|\nabla u|^{-2}\Delta_\infty u$ is negligible when $p$ is close to $p_0=2$ and apply the standard Calder\'on-Zygmund theory. Note that, to do this, we also need information on the summability  of the term $\frac{f}{|\nabla u|^{p-2}}$ that we deduce, following \cite{DS1}, as a consequence of the aforementioned weighted estimate \eqref{kdjlgadklfgbaò}. Such information is provided by the following
\begin{proposition}\label{RHSregpart}
Let  $\Omega$ be a bounded smooth domain of $\mathbb{R}^n,$ and $u$ be weak solution to \eqref{EQ} with $\nabla u\neq \bf{0}$ on $\partial \Omega$. Then, for $p>2$ and for any fixed $1\leq q<\frac{p-1}{p-2},$ there exists a constant $C=C(n,p,q,f)>0$  such that
\begin{equation}\label{drdrdgigregpart}
\int_{\Omega}\frac{f^2}{\Big[|\nabla u|^{p-2}\Big]^q}\leq C.
\end{equation}
The same bound holds for any $1\leq q<+\infty$ when $1<p\leq 2.$
\end{proposition}
This proposition is proved in Section \ref{preliminary results}.
As a consequence of the above proposition we have the following
\begin{theorem}\label{regtheorem}
Let  $\Omega$ be a bounded smooth domain of $\mathbb{R}^n,$ and $u$ be weak solution to \eqref{EQ} with $\nabla u\neq \bf{0}$ on $\partial \Omega$. Let $q\geq 2$ and $p$ be such that \[
|p-2|<\frac{1}{C(n,q)}\,,
\] with $C(n,q)$ given by \eqref{calderon0}.
Then, if $1<p\leq 2,$
there holds
\begin{equation*}
u\in W^{2,q}(\Omega)\,.
\end{equation*}

In the case $p>2,$ the same conclusion holds provided $p$ is in addition such that $q<\frac{p-1}{p-2}.$ \newline
As a consequence, for any $\gamma \in (0,1)$ there exist values $1<p_1<2<p_2$ such that for all $p\in (p_1,p_2)$ there holds $u\in C^{1,\gamma}(\overline{\Omega}).$
\end{theorem}

As an application of the above result we consider
\begin{equation}\label{Ex}
\begin{cases}
-\Delta_p\,u=u^{p^*}-\Lambda u^{p-1} & \text{in}\,\Omega,\\
u \geq 0,&  \text{in}\,\Omega,\\
u=0,&  \text{on}\,\partial\Omega\,,
\end{cases}
\end{equation}
where $\Omega$ is a smooth bounded domain of $\R^n,$ $1<p<n,$ and $p^*=np/(n-p)$ is the critical Sobolev exponent. The case $p=2$ and $n\geq 3$ had been extensively studied since the pioneering papers of Brezis and Nirenberg \cite{breni}, Coron \cite{coron}, and Bahri and Coron \cite{bahricoron}, inspiring
a very broad literature on related existence, non-existence, multiplicity, symmetry, and classification results. The case $p\neq 2$ has several interesting features and related open problems, see e.g. \cite{GV1,GV2} and more recently  \cite{ClTi}, \cite{MeScSq},\cite{MePa}, \cite{MeWi}. A solution $u$ to problem (\ref{Ex}) can be found for suitable $\Lambda$ as nonnegative constrained minimiser, by overcoming well-known lack of compactness phenomena for the minimising (or Palais-Smale) sequences which typically occur in the presence of the critical Sobolev exponent. In particular $u\in L^{\infty}(\Omega)$  by variant of Moser's iteration (see e.g. Appendix E in \cite{Peral}, and \cite{Tr}), and hence  $u\in C^{1,\alpha}(\overline{\Omega}).$ The positivity is due to the maximum principle (\cite{V}, \cite{PSB}). In particular Hopf's boundary point lemma implies that $\nabla u\neq \bf{0}$ on $\partial \Omega, $ therefore $u$ satisfies the hypotheses of Theorem \ref{regtheorem}. \\
We point out that the basic regularity estimate given by Proposition \ref{Hesspart} on the second derivatives up the boundary, is known essentially only in those cases when Hopf's boundary point lemma (see \cite{V}) can be applied, this to rule out the existence of critical points of the solutions along the boundary. For this reason Hopf's boundary point lemma allows to use standard elliptic regularity theory. However there are many cases of interest when the assumptions of Hopf's lemma are not satisfied. In fact, if for instance we consider a sign-changing solution whose nodal line touches the boundary, then necessarily, at the touching point, the gradient of $u$ vanishes. This is the case when the second eigenfunction of the $p$-Laplacian with Dirichlet boundary condition is considered:
\begin{equation}\label{eigen}
\left\{
\begin{array}{lll}

  -\Delta_p u_2\,=\,\lambda_2|u_2|^{p-2}u_2\quad &\mathrm{in}& \,\,\Omega   \\
  \,\,\, u_2=0 &\mathrm{on}& \,\, \partial \Omega.
\end{array}
\right.
\end{equation}

In this case $u_2$ has exactly two nodal regions, as it had been pointed out in \cite{CFG}, see also \cite{lind2}. Motivated by the above problem and similar ones, it is natural to try to obtain regularity results of the same flavour of Theorem \ref{regtheorem} which possibly hold for those equations involving a source term of the form $f(u),$ allowed to be strictly negative near the boundary.\\
To this aim, for $p>1$ and $n\geq 2$ we consider the following boundary value problem.
Let $u$ be a weak solution of:
\begin{equation}\label{eqfu}
\begin{cases}
-\Delta_p\,u=f(u) & \text{ in }\Omega\\
u=0 &  \text{ on }\partial\Omega\,
\end{cases}
\end{equation}
where $\Omega$ be a bounded smooth domain of $\mathbb R^n,$ $f: \R \rightarrow \R$ is continuos, and only when $p>2$ we assume in addition that $f$ is locally Lipschitz satisfying

\begin{eqnarray}
&H1)\nonumber &\qquad \exists \,\gamma>0,\,k>0\,:\, \lim_{t\to 0}\frac{f(t)}{|t|^{k-1}t}=\gamma\\
&H2) \nonumber&\qquad  f(0)=0,\quad f(t)\cdot t>0\quad\forall t\in\R\backslash\{0\}.
\end{eqnarray}

\begin{theorem}\label{regtheorem2}
Let  $\Omega$ be a bounded smooth domain of $\mathbb{R}^n,$ $f$ continuous, and $u \in C^{1,\alpha}(\overline{\Omega})\cap W^{1,p}_0(\Omega)$ be weak solution to \eqref{eqfu}.  Let $q\geq 2$ and $p$ be such that \[
|p-2|<\frac{1}{C(n,q)}\,,
\] with $C(n,q)$ given by \eqref{calderon0}.

Then, if $1<p\leq 2,$
there holds
\begin{equation*}
u\in W^{2,q}(\Omega)\,.
\end{equation*}
In the case $p>2,$ the same conclusion holds provided $f$ is locally Lipschitz and satisfies $H1),H2),$ and $p,q$ and $k$ are such that
\begin{equation}\label{additionalass}
\textrm{max}\Big(\frac{2(k+1)}{k}, q\Big)<\frac{p-1}{p-2}.
\end{equation}\newline
As a consequence, for any $\gamma \in (0,1)$ there exist values $1<p_1<2<p_2$ such that for all $p\in (p_1,p_2)$ there holds $u\in C^{1,\gamma}(\overline{\Omega}).$
\end{theorem}
Going back to our model problem (\ref{eigen}) we can see that Theorem \ref{regtheorem2} is applicable to $u_2,$ where we set $k:=p-1.$ In particular as far as  the degenerate case $p>2$  is concerned, if $q<\frac{2(k+1)}{k}=\frac{2p}{p-1},$ then the condition $\frac{2(k+1)}{k}<\frac{p-1}{p-2}$ holds if and only if $p<1+\sqrt 2.$ Then $u\in W^{2,q}(\Omega),$ provided $p<\min\Big(2+\frac{1}{C(n,q)}, 1+\sqrt2\Big).$ \newline
In order to prove Theorem \ref{regtheorem2} we need a weaker form of equation (\ref{kdjlgadklfgbaò}) which holds in the case $p>2$ when we do not assume $\nabla u\neq \bf{0}$ on $\partial \Omega.$ To this aim we perform a linearisation argument which is in the spirit of \cite{DS1} and \cite{Sci1,Sci2}.

\subsection{Related questions}
There are several related questions which are left open by the present paper to be considered in future projects. We believe that the major ones are the following.
\begin{itemize}
\item [] A) Determining an optimal value for the constant $C(n,q)$ involved in the Calder\'on-Zygmund estimate (\ref{calderon0}) would give sharper regularity results in Theorem \ref{regtheorem} and Theorem \ref{regtheorem2}.
\item [] B) We wonder whether results of similar flavour could be obtained for solutions $u\in W^{1,p}_{\textrm{loc}}(\Omega)$ such that
\begin{equation*}
\int_\Omega |\nabla u|^{p-2}\nabla u\,\nabla \psi\,dx\,=\,\int_\Omega f\psi \,dx\,,
\end{equation*}
for all $\psi\in C^{\infty}_{c}(\Omega)$.
Even the case $f\equiv 0$ of $p$-harmonic functions, especially in higher dimensions $n>2,$ would be significative. In this case one should take into account the effect of the boundary of $\Omega,$ which we do not see in our context because of the Dirichlet boundary condition. This is the reason why the regularising phenomenon we highlight is simply driven by the constant $C(n,q).$
\item [] C) To prove Theorem \ref{regtheorem2} a weaker form of the weighted Hessian estimate has been obtained and used. It would be interesting to check, under the same assumptions of Theorem \ref{regtheorem2}, whether or not the same Hessian estimate (\ref{kdjlgadklfgbaò}) still holds. This could be perhaps accomplished by flattening the boundary and trying to obtain a local version of the same estimate in the spirit of \cite{DS1,Sci1,Sci2}, which is known to the best of our knowledge, only when Hopf's boundary point lemma is applicable. However, many nontrivial technical difficulties  arise with such approach when trying to obtain our summability results on the second derivatives of the solutions.
\item[] D) It would be interesting to weaken our assumptions on $f,$  for instance by considering the non-autonomous case, i.e. the case with a nonlinearity either of the form $f(x,u)$ or more in general $f(x,u, \nabla u).$
\end{itemize}
 \subsection{Organization of the paper}
The paper is organized as follows. In Section \ref{preliminary results} we collect some preliminary results which are useful to prove Theorem \ref{regtheorem}, which we will prove in Section \ref{Appl}; in Section \ref{preliminary results} in particular we prove Proposition \ref{RHSregpart}, which follows from the more general Proposition \ref{RHSreg}.  In Section \ref{Hessian reg sect}, which looks slightly more technical, we derive a Hessian estimate which plays the same role of Proposition \ref{RHSreg} in the new setting of Theorem \ref{regtheorem2}. Finally, in Section \ref{Appl2} we prove Theorem \ref{regtheorem2}.\\
\subsection{Notation} We use the following standard notation:
\begin{itemize}
      \item Limits for sequences of functions $(u_\varepsilon)_{\varepsilon>0}$ as $\varepsilon \rightarrow 0^+$ are meant to be performed for suitable sequences $\varepsilon_i\rightarrow0^+,\, i\rightarrow  \infty.$
	\item $\Delta_p u:={\rm div}(|\nabla u|^{p-2}\nabla u)$  is the classical $p$-Laplacian operator.
	\item $|x|$ is the euclidean norm of $x \in \R^n$.
	\item $\chi_\Omega$ is the characteristic function of a measurable set $\Omega$.
	\item $C^k(\Omega)$ is the space of real valued functions $k$ times continuously differentiable on $\Omega$.
	\item $C^\infty(\Omega)$ is the space of real valued functions  which are infinitely times continuously differentiable on $\Omega$.
	\item $C^k_c(\Omega),$ $C^\infty_c(\Omega)$ are the spaces made up of compactly supported functions of respectively $C^k(\Omega),$ and $C^\infty(\Omega)$
	\item $C^{k,\alpha}(\Omega)$ and $C^{k,\alpha}_{\textrm{loc}}(\Omega)$ are classical H\"older spaces.	
	\item $L^q(\Omega)$and $L^q_{\textrm{loc}}(\Omega)$with $\Omega \subset \R^n$ measurable set and $q \geq 1$, are classical Lebesgue space.
	\item $W^{k,p}(\Omega)$ and $W^{k,p}_{\textrm{loc}}(\Omega)$ are classical Sobolev
	 spaces.
	\item $c_1, ..., c_k, C_1,... C_k, c, c',c''..,C, C',..C''...$ are positive constants.
	\item $\int_{...} ...$ denotes standard Lebesgue integration.	
	\end{itemize}

\section{Preliminaries} \label{preliminary results}
Throughout this section we assume that $\Omega$ is a bounded smooth domain of $\mathbb{R}^n,$ and $u$ is a weak solution to \eqref{EQ} with $\nabla u\neq \bf{0}$ on $\partial \Omega$. \\
We deal now with the proof of Proposition \ref{RHSregpart} which will follow from Proposition \ref{RHSreg} as a particular case ($\varepsilon=0$).\\
\subsection{A classical approximation}
We start observing that in order to deal with the formal expression

\begin{equation}\nonumber
\Delta _p u\,=\, |\nabla u|^{p-2}\Delta u+(p-2)|\nabla u|^{p-4}\Delta_\infty u
\end{equation}
 mentioned in the introduction,  we consider for $\varepsilon\in [0,1)$ the parametric problem
\begin{equation}\label{EQreg}
\begin{cases}
-\text{div}\left(\e+|\nabla u_\e|^2\right)^{\frac{p-2}{2}}\,\nabla u_\e=f, & \text{ in }\Omega\\
u_\e=0.&  \text{ on }\partial\Omega\,.
\end{cases}
\end{equation}


It is standard to see that, when $\varepsilon\in(0,1)$ the above problem regularises the solution to (\ref{EQ}), as by standard regularity theory it follows that there exists a unique solution

\[
u_\e\in C^{2,\alpha} (\overline{\Omega})\cap W^{1,p}_0(\Omega)\,;
\]
$u_\varepsilon$ is therefore a classical solution of
\begin{equation}\label{classic}
-\Delta u_\e\,=\, (p-2)\frac{\left(D^2u_\e\,\nabla u_\e\,,\,\nabla u_\e\right)}{\left(\e+|\nabla u_\e|^2\right)}
\,+\,\frac{f}{\left(\e+|\nabla u_\e|^2\right)^{\frac{p-2}{2}}}
\end{equation}
in $\Omega.$\\


Note that, by \cite{Di,Li,T} we have that
\[
\|u_\e\|_{C^{1,\alpha}(\overline{\Omega})}\leq C\,.
\]
Therefore, passing if necessary to a subsequence, by classical Arzela-Ascoli compactness theorem we have
\begin{equation*}
u_\e\overset{C^{1,\alpha'}(\overline{\Omega})}{\longrightarrow} w\,
\end{equation*}
for some $0<\alpha'<\alpha.$ It follows easily that $w$ is a weak solution to \eqref{EQ} and consequently, by uniqueness of the solution of \eqref{EQ}, we get
\begin{equation}\label{per Fatou}
u_\e\overset{C^{1,\alpha'}(\overline{\Omega})}{\longrightarrow} u\,.
\end{equation}


\subsection{On second derivatives}
\noindent We recall that by the classical Calder\'on-Zygmund theory for elliptic operators we have the following:

\begin{lemma}\label{caldlemma}
Let $\Omega$ be a bounded smooth domain of $\mathbb{R}^n$ and let $w\in W^{2,q}_0(\Omega)$. Then there exists a positive constant $C=C(n,q)$ such that

\begin{equation}\label{calderon}
\|D^2 w\|_{L^q(\Omega)}\leq C\|\Delta w\|_{L^q(\Omega)}.\,
\end{equation}

\end{lemma}
See e.g. Corollary 9.10 in \cite{GiTr}.\\
We will need the following weighted Hessian regularity result.
\begin{proposition}\label{Hess}
Let $u_\varepsilon$ be solution to $(\ref{EQreg})$ and $f$ as in Proposition \ref{Hesspart}.
Then, for any $\beta<1$, there exists a constant $C_\beta=C_\beta(n,p,f)>0$ such that

$$\int_{\Omega}\left(\e+|\nabla u_\e|^2\right)^{\frac{p-2-\beta}{2}}  \|D^2 u_\e\|^2 <C_\beta
$$
for all $\e \in [0,1).$

\end{proposition}
\begin{proof}
The case $\varepsilon =0$ which is the most delicate one, is given by Proposition \ref{Hesspart}. The case $\varepsilon>0$ follows by the same arguments. And this concludes the proof.
\end{proof}
\subsection{Proof of Proposition \ref{RHSregpart}}
The main result of the present section is the following
\begin{proposition}\label{RHSreg}
Let  $\Omega$ be a bounded smooth domain of $\mathbb{R}^n,$ and with $\varepsilon\in[0,1)$ let $u_\e$ be defined by \eqref{EQreg}, and in particular for $\varepsilon=0$ let $u_0:=u$ be the weak solution to \eqref{EQ} with $\nabla u\neq \bf{0}$ on $\partial \Omega$. Then, for $p>2$ and for any fixed $1\leq q<\frac{p-1}{p-2},$ there exists a constant $C=C(n,p,q,f)>0$ which does not depend on $\e$ such that

\begin{equation}\label{drdrdgig}
\int_{\Omega}\frac{f^2}{\Big[\left(\e+|\nabla u_\e|^2\right)^{\frac{(p-2)}{2}}\Big]^q}\leq C.
\end{equation}
The same bound holds for any $1\leq q<+\infty$ when $1<p\leq 2.$
\end{proposition}
We do not assume here $q\geq 2.$
\begin{proof}
For some small $\delta>0$ define $\varphi_\delta=j_\delta*\chi_{K_\delta}$ where $\{j_\delta\}_{\delta>0}$ are standard radial mollifiers, $\{K_\delta\}_{\delta>0}$ is a monotone family of compact smooth subsets of $\Omega$ such that the characteristic functions $\chi_{K_\delta}\rightarrow \chi_\Omega$ as $\delta \rightarrow 0,$ in $L^1(\Omega),$ and
$\Omega_\delta=\{x \in \Omega \, : \, |x-y|\leq\delta, \, \textrm{for some $y\in K_\delta$}\}\subset \Omega.$  By construction $0\leq\varphi_\delta\leq 1,$ $\varphi_\delta\equiv 0$ outside $\Omega_\delta,$ $\lim_{\delta\rightarrow 0}\varphi_\delta=1$ for all $x\in\Omega,$ and zero outside. Moreover by scaling $|\nabla \varphi_\delta|< C/\delta.$ \\
We prove
\begin{equation}\label{deltacut}
\int_{\Omega}\frac{f^2\varphi^2_\delta}{\Big[\left(\e+|\nabla u_\e|^2\right)^{\frac{(p-2)}{2}}\Big]^q}\leq C
\end{equation}
for some uniform constant $C,$  the bound (\ref{drdrdgig}) will then follow by Fatou's lemma as $\delta\rightarrow 0.$ \\
The bounds are deduced for $\e\in (0,1)$ as the case $\e=0$ is a consequence again of Fatou's lemma, by using (\ref{per Fatou}) and the fact that the estimates are uniform in $\e.$
We test equation (\ref{EQreg}) against

\[
\psi\,:=\frac{f\varphi_\delta^2 }{\Big[\left(\e+|\nabla u_\e|^2\right)^{\frac{(p-2)}{2}}\Big]^q}.
\]

In fact using (\ref{EQreg}) we estimate

\begin{equation}\label{membrodestro}
\begin{split}
\int_{\Omega}\frac{f^2\varphi^2_\delta}{\Big[\left(\e+|\nabla u_\e|^2\right)^{\frac{(p-2)}{2}}\Big]^q}\leq& \Big| \int \left(\e+|\nabla u_\e|^2\right)^{\frac{(p-2)}{2}}\left(\nabla u_\e\,, \nabla \psi \right)\Big| \\
& \leq C_0 (A_\e+B_\e+C_\e),
\end{split}
\end{equation}
where we have set

\begin{equation}\nonumber
\begin{split}
A_\e &=\int_{\Omega}\left(\e+|\nabla u_\e|^2\right)^{\frac{(p-2)}{2}(1-q)}|\nabla u_\e| \cdot  \varphi_\delta \cdot | \nabla \varphi_\delta|\cdot |f|\\
B_\e & =\int_{\Omega}\left(\e+|\nabla u_\e|^2\right)^{\frac{(p-2)}{2}(1-q)}|\nabla u_\e|\cdot |\varphi_\delta|^2\cdot |\nabla f| \\
C_\e & =\int_{\Omega}\left(\e+|\nabla u_\e|^2\right)^{\frac{(p-2)}{2}(1-q)-1}|\nabla u_\e|^2\cdot \|D^2 u_\e\| \cdot |f|\cdot |\varphi_\delta|^2.
\end{split}
\end{equation}

We handle each term separately. \\

1) Estimate on $A_\e.$ The restriction of $q$ and the fact that $\|u_\e\|_{C^{1,\alpha}}\leq C_1$ uniformly with respect to $\e$ yields $\|\left(\e+|\nabla u_\e|^2\right)^{\frac{(p-2)}{2}(1-q)}|\nabla u_\e| \|_{\infty}\leq C_2.$
Notice that, in the construction of
$\{K_\delta\}_{\delta>0}$ and $\Omega_\delta$, we can assume with no loss of generality that   $|\textrm{supp}\nabla \varphi_\delta|<C_3\delta,$ for some constant $C_3>0$. Therefore  it follows the uniform bound
$$A_\e \leq C_4 \int_{\textrm{supp}\nabla \varphi_\delta}  | \nabla \varphi_\delta|\leq C_5 \frac{|\textrm{supp}\nabla \varphi_\delta|}{\delta}\leq C_6.$$
2) Estimate on $B_\e.$ Again we use that $\|\left(\e+|\nabla u_\e|^2\right)^{\frac{(p-2)}{2}(1-q)}|\nabla u_\e| \|_{\infty}$ is uniformly bounded together with the fact that $|\nabla f|\in L^1(\Omega)$, to conclude that $$B_\e\leq C_7.$$
3) Estimate on $C_\e.$ We obviously have $$C_\e\leq \int_{\Omega}\frac{|f|\varphi_\delta}{\Big[\left(\e+|\nabla u_\e|^2\right)^{\frac{(p-2)}{4}}\Big]^q}\cdot  \left(\e+|\nabla u_\e|^2\right)^{\frac{(p-2)}{2}-\frac{(p-2)}{4}q} \|D^2 u_\e\|. $$

By the elementary inequality on positive numbers $ab\leq \eta a^2+\frac{b^2}{\eta}$ it follows that
$$
C_\e \leq \eta \int_{\Omega}\frac{f^2\varphi^2_\delta}{\Big[\left(\e+|\nabla u_\e|^2\right)^{\frac{(p-2)}{2}}\Big]^q}+\frac{1}{\eta}\int_{\Omega}\left(\e+|\nabla u_\e|^2\right)^{p-2-\frac{(p-2)}{2}q} \|D^2 u_\e \|^2.
$$
In view of the restriction on $q$ we can apply Proposition \ref{Hess}, obtaining
$$
C_\e \leq \eta \int_{\Omega}\frac{f^2\varphi^2_\delta}{\Big[\left(\e+|\nabla u_\e|^2\right)^{\frac{(p-2)}{2}}\Big]^q}+\frac{1}{\eta}C_{8}
$$
Conclusion. By using the above estimates on $A_\e,B_\e,C_\e$, equation (\ref{membrodestro}) yields for some $\eta <\frac{1}{C_0}$

$$(1-\eta C_0 )\int_{\Omega}\frac{f^2\varphi^2_\delta}{\Big[\left(\e+|\nabla u_\e|^2\right)^{\frac{(p-2)}{2}}\Big]^q} \leq C(\eta)$$ where the constant $C(\eta)$
does not depend on $\e.$
 This concludes the proof.
\end{proof}

\section{Proof of Theorem \ref{regtheorem}}\label{Appl}
With the results of the preceding sections at hand we are now in position to prove Theorem \ref{regtheorem}.\begin{proof} [Proof of Theorem \ref{regtheorem}]
Let us start considering the case $p>2$.
Note that, as already recalled in Section \ref{preliminary results}, passing if necessary to a subsequence we have that
\begin{equation}\label{hdjshdjshjjsdfghdeyrty11}
u_\eps\overset{C^{1,\alpha'}(\overline{\Omega})}{\longrightarrow} u.\,
\end{equation}

Since  $2\leq q<\frac{p-1}{p-2},$ by Proposition \ref{RHSreg}, Proposition \ref{Hess}, by equations (\ref{classic}) and (\ref{calderon}) we deduce that
\begin{equation}\nonumber
\begin{split}
\left\|D^2 u_\e\right\|_{L^q(\Omega)}&\leq C(n,q)\left\|(p-2)\frac{\left(D^2u_\e\,\nabla u_\e\,,\,\nabla u_\e\right)}{\left(\e+|\nabla u_\e|^2\right)}
\,+\,\frac{f}{\left(\e+|\nabla u_\e|^2\right)^{\frac{p-2}{2}}}\right\|_{L^q(\Omega)}\\
&\leq C(n,q)(p-2)\left\|D^2 u_\e\right\|_{L^q(\Omega)}+C(n,q)\left\|\frac{f}{\left(\e+|\nabla u_\e|^2\right)^{\frac{p-2}{2}}}\right\|_{L^q(\Omega)}\\
&\leq  C(n,q)(p-2)\left\|D^2 u_\e\right\|_{L^q(\Omega)}+\tilde C(n,p,q,f)
\end{split}
\end{equation}
where we used that $2\leq q<\frac{p-1}{p-2}$. Here $\tilde C(n,p, q, f)= C\cdot C(n,q)$ where $C$ is given by Proposition \ref{RHSreg} and we also have used that $f$ is bounded by classical Morrey's embedding theorem. \\It follows
\begin{equation}\nonumber
\begin{split}
(1-C(n,q)(p-2))\left\|D^2 u_\e\right\|_{L^q(\Omega)}&\leq \tilde C(n,p,q, f)\,.
\end{split}
\end{equation}
Since $p-2<\frac{1}{C(n,q)}$ there holds $$\sup_{\e> 0} \left\|u_\e\right\|_{W^{2,q}(\Omega)}<\infty.$$
Classical Rellich's theorem implies now that passing if necessary to a subsequence
\[
u_\e\rightharpoonup w\in W^{2,q}(\Omega),\qquad \text{and }\, \textrm{almost everywhere in}\,\,\Omega\,.
\]
Therefore we infer that
\[
u\equiv w\in W^{2,q}(\Omega)\,.
\]

\noindent The proof in the case $1<p<2$ can be carried out exactly in the same way and observing that \eqref{drdrdgig}
 is obvious, being non-singular, therefore Proposition \ref{RHSreg} is not needed in this case. \\
Let now $\gamma \in (0,1)$ be fixed. There exists $q$ such that $$\gamma=1-\frac{n}{q},$$ and by the preceding part of the proof $u\in W^{2,q}(\Omega),$ for all $p$ in a suitable open interval containing
$p_0=2.$  It follows by classical Morrey's embedding that $$\partial_i u \in C^{0,\gamma}(\overline{\Omega}),\,\, i=1,...n.$$
This concludes the proof.

\end{proof}

\section{The autonomous equation}\label{Hessian reg sect}
\subsection{The approximated equation}
Let $\Omega$ be a bounded smooth domain of $\mathbb R^n,$ and $p>1.$ We consider now the autonomous equation
\begin{equation}\label{eq forte p-lap fu}
\begin{cases}
-\Delta_p\,u=f(u) & \text{ in }\Omega,\\
u=0 &  \text{ on }\partial\Omega\,
\end{cases}
\end{equation}
where $f:\R\rightarrow \R$ is a continuous function and $u\in C^{1,\alpha}(\overline{\Omega})$ is a weak solution.
We adapt here the approximation argument which has been used in the preceding section.
To this aim we consider the equation:
\begin{equation}\label{eq fu forte}
\begin{cases}
-\text{div}\left(\e+|\nabla u_\e|^2\right)^{\frac{p-2}{2}}\,\nabla u_\e=f(u)& \text{ in }\Omega, \quad \varepsilon\in (0,1),\\
u_\e=0 &  \text{ on }\partial\Omega.
\end{cases}
\end{equation}

Obviously $u_\e\in C^2(\overline{\Omega})$ defined by (\ref{eq fu forte}) is such that
\begin{equation}\label{equazione eps debole}
\int_\Omega \left(\varepsilon+|\nabla u_\varepsilon|^2
\right)^{\frac{p-2}{2}}\nabla u_\varepsilon\,\nabla
\psi\,dx\,=\,\int_\Omega f(u)\psi \,dx\,,
\end{equation}
for all $\psi\in C^\infty_c(\Omega)$.\\
Moreover
$u_\varepsilon$ is a classical solution of
\begin{equation}\label{classic2}
-\Delta u_\e\,=\, (p-2)\frac{\left(D^2u_\e\,\nabla u_\e\,,\,\nabla u_\e\right)}{\left(\e+|\nabla u_\e|^2\right)}
\,+\,\frac{f(u)}{\left(\e+|\nabla u_\e|^2\right)^{\frac{p-2}{2}}}
\end{equation}
in $\Omega.$\\

Although (\ref{eq forte p-lap fu}) may certainly have multiple solutions, uniqueness holds for equation (\ref{eq fu forte}), as well as for
\begin{equation}\label{eq forte p-lap fu 2}
\begin{cases}
-\Delta_p\,v=f(u) & \text{ in }\Omega\\
v=0 &  \text{ on }\partial\Omega\,.
\end{cases}
\end{equation}
By these observations, arguing exactly as in the preceding section again by the classical Arzela-Ascoli compactness theorem and up to a subsequence, there holds
\begin{equation*}
u_\e\overset{C^{1,\alpha'}(\overline{\Omega})}{\longrightarrow} u\,
\end{equation*}
for some $0<\alpha'<\alpha.$ \\

\subsection{Hessian estimates in the degenerate case $p>2$}
Throughout the present section $p>2;$ moreover the subscript $i$ indicates
the derivative with respect to $x_i:$ $u_i= \frac{\partial u}{\partial x_i}, \,\,i=1,...n.$
We also assume $f:\R\rightarrow \R$ to be a locally Lipschitz continuos function such that:
\begin{equation}\label{f propr1}
\exists \,\gamma>0,\,k>0\,:\, \lim_{t\to 0}\frac{f(t)}{|t|^{k-1}t}=\gamma
\end{equation}
and
\begin{equation}\label{f propr2}
f(0)=0,\quad f(t)\cdot t>0\quad\forall t\in\R\backslash\{0\}.
\end{equation}
In the following we will use (\ref{equazione eps debole}), and as in the preceding section, we will again make a suitable choice for a test function which after integrating by parts, with some abuse of language, linearises equation (\ref{equazione eps debole}).\\
For any $i=1,...,n$, plugging
$\varphi_i$ as test function into (\ref{equazione eps debole}) and
integrating by parts, we
obtain
\begin{eqnarray}\label{linearizzato}
  && \io\neps^{\frac{p-2}{2}}\la \n \uei,\n\vf\ra+
  (p-2)\io\neps^{\frac{p-4}{2}}\la\n \uei,\n \ue\ra \la\n \ue,\n\vf\ra= \\
\nonumber &=& \io f'(u)u_i\vf\,,
\end{eqnarray}
for all $\vf \in
C^1_c(\Omega)$ and by density for all $\vf \in
W^{1,1}_0(\Omega).$\\We have the following Hessian estimate which plays the same role of equation (\ref{kdjlgadklfgbaò}) given in Proposition \ref{Hesspart},
which is suitable in the present autonomous case, without assuming $\nabla u\neq\bf{0}$ on $\partial \Omega.$
\begin{proposition}[Weighted Hessian estimate]\label{stima hessiano}
Let $u_\eps$ be as above. For all $\beta\in [0,1)$, there holds:
\begin{equation}\label{stima hessiano 2}
\io\neps^\frac{p-2-\beta}{2} | D^2 u_\eps |^2 u^2  \leq C\,,
\end{equation}
where $C= C(\beta,p,n,f)>0$ is independent on $\eps$. The same estimate holds for $\varepsilon=0$ and $u_0:=u.$
\end{proposition}
\proof Let $G_\xi:\R\to\R$ be defined as
$$
G_\xi(s)=\left\{
\begin{array}{ll}
s & \hbox{ if } |s| \geq 2 \xi, \\
2\left[s- \xi \frac{s}{|s|}\right]  & \hbox{ if }  \xi< |s|< 2\xi, \\
0 & \hbox{ if } |s|\leq \xi.
\end{array}
\right.
$$
Fix $\beta\in [0,1)$ and set
$$
T_\xi(t)=\frac{G_\xi(t)}{|t|^\beta}
$$
and
\begin{equation}\label{test}
\vf (x)=T_\xi(u_{\eps,i}(x))u^2(x)
\end{equation}
In the sequel we omit
the dependence on $x.$  Using $\vf$ as
test function in \eqref{linearizzato}, we have
\begin{eqnarray}\label{linearizzato con fi part}
  && \\
\nonumber && \io\neps^{\frac{p-2}{2}}|\n \uei|^2 T'_\xi(\uei)u^2+
   \io\neps^{\frac{p-2}{2}}\la\n \uei,\n u\ra T_\xi(\uei)2u\\
\nonumber &+&(p-2)\io \neps^{\frac{p-4}{2}}\la\n \uei,\n u_\e\ra^2
T'_\xi(\uei)u^2\\
\nonumber &+&
   (p-2)\io\neps^{\frac{p-4}{2}}\la\n \uei,\n u_\e\ra \la\n u_\e,\n u\ra T_\xi(\uei)2u\\
\nonumber &=&\io f'(u)u_i|T_\xi(\uei)|u^2.
\end{eqnarray}
In the sequel $c$ and $C$ will denote  positive constants,
possibly depending on $\|u_\eps\|_{W^{1,\infty}}$, whose value can
vary from line to line.

We set

\begin{eqnarray}
I_1 &=& \io\neps^{\frac{p-2}{2}}|\n \uei|^2 T'_\xi(\uei)u^2\\
\nonumber I_2 &=& (p-2)\io \neps^{\frac{p-4}{2}}\la\n \uei,\n u_\e\ra^2 T'_\xi(\uei)u^2\\
\nonumber I_3 &=& \io\neps^{\frac{p-2}{2}}\la\n \uei,\n u\ra T_\xi(\uei)2u\\
\nonumber I_4 &=& (p-2)\io\neps^{\frac{p-4}{2}}\la\n \uei,\n u_\e\ra \la \n u_\e,\n u\ra T_\xi(\uei)2u\\
\nonumber I_5 &=& \io f'(u)u_i T_\xi(\uei)u^2.
\end{eqnarray}

If $p\geq 2$, then $I_2$ is positive and hence we have
\begin{equation}\label{I1I2 a}
I_1+I_2\geq I_1\,.
\end{equation}

By \eqref{linearizzato con fi part} and \eqref{I1I2 a} we infer
\begin{equation}\label{diseq linearizzato 1}
I_1\leq |I_3|+|I_4|+|I_5|
\end{equation}

and hence, recalling that $u\in C^1(\overline\Omega)$, we get

\begin{equation}\label{stima 1}
\io\neps^{\frac{p-2}{2}}|\n \uei|^2 T'_\xi(\uei)u^2 \leq
c\io\neps^{\frac{p-2}{2}}|\n \uei||T_\xi(\uei)|u|+
\io|f'(u)||u_i||T_\xi(\uei)|u^2.
\end{equation}

We recall that there exists $M>0$ such that
\begin{equation}\label{equilimitatezza norme u}
\sup_{\eps}||u_\eps||_{W^{1,\infty}}\leq M\,.
\end{equation}

Therefore, recalling that $f$ is locally Lipschitz continuous, by
\eqref{equilimitatezza norme u} we get
\begin{equation}\label{stima pezzo con f}
\io|f'(u)||u_i||T_\xi(\uei)|u^2\leq C.
\end{equation}

Using \eqref{equilimitatezza norme u} and the elementary inequality
$ab\leq\theta a^2+\frac{1}{4\theta}b^2$  (for all $a,b\in\mathbb
R$ and $\theta>0$), we have
\begin{eqnarray}
&&\io\neps^{\frac{p-2}{2}}|\n \uei||T_\xi(\uei)|u|\\
\nonumber
  &=&\io
  \frac{\neps^{\frac{p-2}{4}}|\n \uei|G_\xi(\uei)^{\frac{1}{2}}|u|}{|\uei|^{\frac{\beta}{2}}|\uei|^{\frac{1}{2}}}
  \cdot
  \frac{\neps^{\frac{p-2}{4}}G_\xi(\uei)^{\frac{1}{2}}|\uei||\n u_\e|}{|\uei|^{\frac{\beta}{2}}|\uei|^{\frac{1}{2}}}\\
\nonumber
  &\leq&
  \theta\io
  \frac{\neps^{\frac{p-2}{2}} |\n \uei|^2 G_\xi(\uei)u^2 }{|\uei|^\beta \uei}
  +\frac{1}{4\theta}\io
  \neps^{\frac{p-\beta+2}{2}}G_\xi(\uei)\\
\nonumber
  &\leq&
  \theta\io
  \frac{\neps^{\frac{p-2}{2}} |\n \uei|^2 G_\xi(\uei)u^2 }{|\uei|^\beta \uei}
  +C.
\end{eqnarray}

Since
$$
T'_\xi(s)=\frac{1}{|s|^\beta}\left[G'_\xi(s)-\beta\frac{G_\xi(s)}{s}\right],
$$
after setting $\vartheta=c\theta$, by \eqref{stima 1} we get
\begin{equation}\label{stima hess 1}
\io \frac{\neps^\frac{p-2}{2}|\n \uei|^2}{|\uei|^\beta}
\left(G'_\xi(\uei)-(\beta+\vartheta)\frac{G_\xi(\uei)}{\uei}\right)u^2
\leq C.
\end{equation}

Choosing $\vartheta$ such that $\beta+\vartheta<1$, we have that
$G'_\xi(w^\eps_j)-(\beta+\vartheta)\frac{G_\xi(\uei)}{\uei}$ is
positive and by definition of $G_\xi$ it follows that

$$
G'_\xi(s)-(\beta+\vartheta)\frac{G_\xi(s)}{s}\to 1-(\beta+\theta)
$$
as $\xi\rightarrow 0$ and hence by Fatou's Lemma
\begin{equation}\label{stima hess 2}
\int_{\Omega\setminus\{\uei=0\}} \frac{\neps^\frac{p-2}{2}|\n
\uei|^2}{|\uei|^\beta}u^2\leq C\,.
\end{equation}

Moreover, since $|\uei|\leq |\n u_\eps|$, we have
\begin{eqnarray}
\nonumber \int_{\Omega\setminus\{\uei=0\}}
\neps^\frac{p-2-\beta}{2}|\n \uei|^2u^2 &=&
\int_{\Omega\setminus\{\uei=0\}}
\frac{\neps^\frac{p-2}{2}|\n \uei|^2 u^2}{\neps^\frac{\beta}{2}}\\
\nonumber&\leq& \int_{\Omega\setminus\{\uei=0\}}
\frac{\neps^\frac{p-2}{2}|\n \uei|^2 u^2}{|\uei|^\beta }
\end{eqnarray}
and hence by \eqref{stima hess 2} it follows
\begin{equation}\label{figata}
\int_{\Omega\setminus\{\uei=0\}} \neps^\frac{p-2-\beta}{2}|\n
\uei|^2 u^2\leq C
\end{equation}
where $C$ depends on $n,p,\beta,f$.
Since $\n \uei=0$ almost everywhere on
$\{\uei=0\}$ the statement is proved for all $\varepsilon\in (0,1).$\\
To prove the statement for $\varepsilon=0$ we argue as follows. Observe that since $\n u_\varepsilon \rightarrow \n u$ uniformy and by elliptic regularity theory, we have
$$u_\varepsilon \rightarrow u$$ in some $C^{2,\overline{\alpha}}(\overline{\omega})$ for all $\omega$ strictly contained in
$ \Omega\setminus Z_u$, where $Z_u=\{x\in \overline{\Omega} \, :\,\nabla u= \bf{0}\}$ is the critical set of $u$. Therefore, for some sequence, $\partial_{i,j}u_\varepsilon \rightarrow \partial_{i,j} u$ almost everywhere on $\Omega\setminus Z_u.$  Finally by Fatou's lemma we get from (\ref{figata})
\begin{equation}\label{stima hess 3}
 \int_{\Omega\setminus Z_u} |\n u|^{p-2-\beta}|\n u_i|^2
u^2 \leq C\,,
\end{equation}
and the estimate holds on the whole $\Omega,$ as by Stampacchia theorem $\partial_{i,j} u$ vanish almost everywhere on $Z_u.$
This concludes the proof.
\endproof
We now handle the summability of singularly weighted integrals involving $f.$ The following statements are obvious for $p\leq 2.$

\begin{proposition}[Singularly weighted estimate]\label{propo main estimate}
For $p> 2$ let $u_\e$ be given by \eqref{eq fu forte}.
Let $s$ and $p$ be such that $1\leq s<\frac{p-1}{p-2}$.
Then there exists a positive constant $C=C(p,n,f)$, independent on $\e$,
such that:
\begin{equation}\label{eq main estimate}
\io\frac{|u|^{2(k+1)}}{\neps^{\frac{p-2}{2}s}}
\leq
C\,.
\end{equation}
The same estimate holds for $\varepsilon=0$ and $u_0:=u.$\end{proposition}
\proof
We use
$$
\psi=\frac{|u|^{k+1}u}{\neps^{\frac{p-2}{2}s}}
$$
as test function in \eqref{equazione eps debole}.
After setting $L=||u||_{\infty}$,
\eqref{f propr1} and \eqref{f propr2} imply that there exists
$\lambda'>0$ such that
\begin{equation}\label{f propr3}
\frac{f(t)}{|t|^{k-1}t}\geq \lambda'\quad \textrm{for all} \,\, 0<|t|\leq L.
\end{equation}
Hence we have:
\begin{eqnarray}
\lambda' \nonumber\io\frac{u^{2(k+1)}}{\neps^{\frac{p-2}{2}s}} &\leq& \io f(u)\psi\\
\nonumber &\leq&
  c\io\frac{\neps^{\frac{p-2}{2}}|\n \ue|^2|D^2 \ue||u|^{k+2}}{\neps^{\frac{p-2}{2}s+1}}\\
\nonumber &\quad &\qquad
  +c\io\frac{\neps^{\frac{p-2}{2}}|\n \ue||\n u||u|^{k+1}}{\neps^{\frac{p-2}{2}s}}\\
\nonumber &\leq&
  c\io\neps^{\frac{p-2}{2}(1-s)-1}|\n \ue|^2|D^2 \ue||u|^{k+2}\\
 \nonumber &\quad &\qquad +c\io\neps^{\frac{p-1-(p-2)s}{2}}|u|^{k+1}\\
\nonumber &\leq&
  c\io\neps^{\frac{p-2}{2}(1-s)-1}|\n \ue|^2|D^2 \ue||u|^{k+2}\\
  \nonumber &\quad &\qquad+C\\
\nonumber &=&
  c\io\frac{|u|^{k+1}}{\neps^{\frac{p-2}{4}s}}
  \neps^{\frac{p-2}{4}s+\frac{p-2}{2}(1-s)-1}|\n \ue|^2|D^2 \ue||u|\\
\nonumber &\quad &\qquad+ C\\
\nonumber \textrm{(Young's inequality)}&\leq&
  c\theta\io\frac{|u|^{2(k+1)}}{\neps^{\frac{p-2}{2}s}}\\
  \nonumber &\quad &\qquad+\frac{c}{4\theta}\io\neps^{\frac{p-2}{2}s+(p-2)(1-s)-2}|\n\ue|^4|D^2 \ue|^2|u|^2.
\end{eqnarray}
After setting $\vartheta=c\theta$, we have:
\begin{equation}
(\lambda'-\vartheta)\io\frac{|u|^{2(k+1)}}{\neps^{\frac{p-2}{2}s}}
\leq
\io\neps^{\frac{p-2-(p-2)(s-1)}{2}}|D^2 \ue|^2|u|^2
\end{equation}
and, recalling that $s<\frac{p-1}{p-2}$, we can apply
Proposition \ref{stima hessiano} with $\beta=(p-2)(s-1).$
The statement for $\varepsilon=0$ follows by Fatou's lemma, and this conclude the proof.
\endproof

\begin{proposition}[Singularly weighted estimate for $f$]\label{stima destra}
Let $p> 2$ let $u_\e$ be given by \eqref{eq fu forte},
and let $r\geq 1, k>0$ and $p$ be such that
\begin{equation}\label{max}\textrm{max}\Big(\frac{2(k+1)}{k}, r\Big)<\frac{p-1}{p-2}.
\end{equation}Then there exists a positive constant $C=C(p,n,f)$, independent on $\e$,
such that:
\begin{equation}
\io\left(
\frac{|f(u)|}{\neps^{\frac{p-2}{2}}}
\right)^r
\leq C\,.
\end{equation}
The same estimate holds for $\varepsilon=0$ and $u_0:=u.$
\end{proposition}

\proof

After setting $L=||u||_{\infty}$,
\eqref{f propr1} and \eqref{f propr2} imply that there exists
$\lambda'>0$ such that:
\begin{equation}\label{f propr4}
|f(t)|\leq \lambda'|t|^{k}\quad \textrm{for all} \,\, |t|\leq L.
\end{equation}

Since (\ref{max}) holds we have two cases: $r\geq \frac{2(k+1)}{k}$ and $r<\frac{2(k+1)}{k} .$ \\
If $r\geq \frac{2(k+1)}{k}$ by \eqref{f propr4} we have:
\begin{equation}\label{stima Lr}
\io\left(
\frac{|f(u)|}{\neps^{\frac{p-2}{2}}}
\right)^r
\leq
c\io\frac{|u|^{kr}}{\neps^{\frac{p-2}{2}r}}
\leq
c'
\io\frac{|u|^{2(k+1)}}{\neps^{\frac{p-2}{2}r}}
\end{equation}
and the conclusion follows by \eqref{eq main estimate} taking
$s=r.$
If $r<\frac{2(k+1)}{k} ,$ we set $s=\frac{2(k+1)}{k}$ and by H\"older inequality we estimate
\begin{equation}\label{stima Lrr}
\io\left(
\frac{|f(u)|}{\neps^{\frac{p-2}{2}}}
\right)^r
\leq
|\Omega|^{1-\frac{r}{s}}\left(\io\left(
\frac{|f(u)|}{\neps^{\frac{p-2}{2}}}
\right)^s\right)^{\frac{r}{s}}
\end{equation}
and the right hand side is uniformly bounded because of the preceding case. \\
Finally the case $\varepsilon=0$ follows as in the preceding proof by Fatou's lemma, and this concludes the proof.
\endproof

\section{Proof of Theorem \ref{regtheorem2}}\label{Appl2}
In the present section we prove Theorem \ref{regtheorem2} following the same scheme of the proof of Theorem \ref{regtheorem}.
\begin{proof} [Proof of Theorem \ref{regtheorem2}]
Again we consider first the case $p>2.$
Note that, as already observed earlier in the preceding section,
\begin{equation}\label{hdjshdjshjjsdfghdeyrty11}
u_\eps\overset{C^{1,\alpha'}(\overline{\Omega})}{\longrightarrow} u.\,
\end{equation}
By Proposition \ref{stima hessiano}, Proposition \ref{stima destra} with $r=q$, (\ref{classic2}) and (\ref{calderon}) we deduce that
\begin{equation}\nonumber
\begin{split}
\left\|D^2 u_\e\right\|_{L^q(\Omega)}&\leq C(n,q)\left\|(p-2)\frac{\left(D^2u_\e\,\nabla u_\e\,,\,\nabla u_\e\right)}{\left(\e+|\nabla u_\e|^2\right)}
\,+\,\frac{f}{\left(\e+|\nabla u_\e|^2\right)^{\frac{p-2}{2}}}\right\|_{L^q(\Omega)}\\
&\leq C(n,q)(p-2)\left\|D^2 u_\e\right\|_{L^q(\Omega)}+C(n,q)\left\|\frac{f}{\left(\e+|\nabla u_\e|^2\right)^{\frac{p-2}{2}}}\right\|_{L^q(\Omega)}\\
&\leq  C(n,q)(p-2)\left\|D^2 u_\e\right\|_{L^q(\Omega)}+\tilde C.
\end{split}
\end{equation}
Here $\tilde C= \tilde C(p, q, n, f)=C \cdot C(n,q)$ where $C$ is given by Proposition \ref{stima destra}. It follows
\begin{equation}\nonumber
\begin{split}
(1-C(n,q)(p-2))\left\|D^2 u_\e\right\|_{L^q(\Omega)}&\leq \tilde C\,.
\end{split}
\end{equation}
Since $p-2<\frac{1}{C(n,q)}$ we have that $$\sup_{\e> 0} \left\|u_\e\right\|_{W^{2,q}(\Omega)}<\infty.$$
Classical Rellich's theorem implies now that up to subsequences
\[
u_\e\rightharpoonup w\in W^{2,q}(\Omega),\qquad \text{and }\, \textrm{almost everywhere in}\,\,\Omega\,.
\]
Therefore we have that
\[
u\equiv w\in W^{2,q}(\Omega)\,.
\]

\noindent The proof in the case $1<p<2$ can be carried out exactly in the same way observing that Proposition \ref{stima destra}
is not needed, as weighted integrals are non-singular in this case. \\
The statement on the $C^{1,\gamma}$ regularity follows by the same argument used for Theorem \ref{regtheorem}. And this concludes the proof.

\end{proof}

\section*{Acknowledgements}
C.M. would like to thank the Department of Mathematics and Computer Science - Universit\`a
 della Calabria (Italy) for the warm hospitality when the present paper has been written.

\bigskip

\end{document}